\documentclass[11pt,a4paper]{amsart}
\usepackage{amssymb,xspace}
\usepackage{amstext}
\usepackage{tikz}
\usetikzlibrary{arrows,decorations.pathmorphing,backgrounds,fit,positioning,shapes.symbols,chains}
\theoremstyle{plain}
\usepackage{amsbsy,amssymb,amsfonts,latexsym}

\usepackage{enumerate}
\usepackage{graphics}

\date{\today}
\title{Multifractal spectra of typical and prevalent measures}
\author{Fr\'ed\'eric Bayart}
\address{
Clermont Universit\'e, Universit\'e Blaise Pascal, Laboratoire de Math\'ematiques, BP 10448, F-63000 CLERMONT-FERRAND -
CNRS, UMR 6620, Laboratoire de Math\'ematiques, F-63177 AUBIERE
}
\email{Frederic.Bayart@math.univ-bpclermont.fr}

\subjclass{}

\keywords{}

\newcommand{\veps}{\varepsilon}

\def\RR{\mathbb R}
\def\NN{\mathbb N}

\def\P{\mathcal P}
\def\M{\mathcal M}
\def\H{\mathcal H}

\def\E{\mathcal E}
\def\R{\mathcal R}
\def\Ms{\mathcal M^s}

\def\dimh{\dim_{\mathcal H}}
\def\dimp{\dim_{\mathcal P}}
\def\dinfloc{\underline{\dim}_{\rm loc}}
\def\dsuploc{\overline{\dim}_{\rm loc}}
\def\dinflocmux{\underline{\dim}_{\rm loc}(\mu;x)}
\def\dsuplocmux{\overline{\dim}_{\rm loc}(\mu;x)}
\def\dlocmux{\dim_{\rm loc}(\mu;x)}
\def\dinfmu{\underline{D}_\mu}
\def\dinfmuq{\underline{D}_\mu(q)}
\def\Emoinsmualpha{\E_-(\mu;\alpha)}
\def\Eplusmualpha{\E^+(\mu;\alpha)}
\def\emoinsmualpha{E_-(\mu;\alpha)}
\def\eplusmualpha{E^+(\mu;\alpha)}
\def\emualpha{E(\mu;\alpha)}
\def\pk{\mathcal P(K)}
\def\fk{\mathcal F(K)}
\def\dboxsup{\overline{\dim}_B}
\def\dboxsuploc{\overline{\dim}_{B,{\rm loc}}}
\def\Ulm{\mathcal U_{l,m}}
\def\bfi{{\mathbf i}}
\def\bfp{{\mathbf p}}

\DeclareMathOperator{\supp}{supp}

\newtheorem{theorem}{Theorem}[section]

\newtheorem{lemma}[theorem]{Lemma}

\newtheorem{proposition}[theorem]{Proposition}

\newtheorem{corollary}[theorem]{Corollary}

{\theoremstyle{definition}}
{\theoremstyle{definition}}

{\theoremstyle{definition}}

{\theoremstyle{definition}}

{\theoremstyle{definition}}

{\theoremstyle{definition}\newtheorem{remark}[theorem]{Remark}}

\newtheorem*{BS}{Theorem (Buczolich-Seuret)}

\begin{document}

\begin{abstract}
We prove that, in the Baire category sense,  a typical measure supported by a compact set 
admits a linear lower singularity spectrum. We investigate the same question for the upper singularity spectrum
and for other forms of genericity.
\end{abstract}

\maketitle

\section{Introduction}
The \emph{lower} and \emph{upper local dimension} of $\mu$ at $x$ are defined by 
\begin{eqnarray*}
\dinflocmux&=&\liminf_{r\to 0}\frac{\log\mu\big(B(x,r)\big)}{\log r}\\
\dsuplocmux&=&\limsup_{r\to 0}\frac{\log\mu\big(B(x,r)\big)}{\log r}.
\end{eqnarray*}
When these two dimensions coincide, we say that $\mu$ admits a local dimension at $x$ and we denote this common value
$\dlocmux$.

The study of pointwise dimension maps and their relationship to Hausdorff dimension goes back to Bilingsley (\cite{Bil60}, \cite{Bil61}). Setting 
\begin{eqnarray*}
\Emoinsmualpha&=&\{x\in K;\ \dinflocmux\leq\alpha\}\\
\Eplusmualpha&=&\{x\in K;\ \dsuplocmux\leq\alpha\},
\end{eqnarray*}
it is well known that 
\begin{eqnarray*}
\dim_\H(\Emoinsmualpha)&\leq&\alpha\quad\quad\textrm{ for any }\alpha\in[0,\dimh(K)]\\
\dimp(\Eplusmualpha)&\leq&\alpha\quad\quad\textrm{ for any }\alpha\in[0,\dimp(K)],
\end{eqnarray*}
where $\dimh(E)$ (resp. $\dimp(E)$) denotes the Hausdorff (resp. the packing) dimension of $E$. By appropriate versions of the Frostman lemma, 
one may easily prove that these inequalities are optimal (we refer to the classical books \cite{Fal97,Mat95} for the usual definitions and results which are involved in this paper).

To provide a more in depth study, we can carry out a multifractal analysis of the measure $\mu$.
 We then introduce the level sets of the local dimension maps by setting 
\begin{eqnarray*}
\emoinsmualpha&=&\{x\in K;\ \dinflocmux=\alpha\}\\
\eplusmualpha&=&\{x\in K;\ \dsuplocmux=\alpha\}\\
\emualpha&=&\{x\in K;\ \dlocmux=\alpha\}.
\end{eqnarray*}
These sets describe the distribution of the singularities of the measure $\mu$ and thus contain crucial informations on the geometrical properties
of $\mu$. The maps $\alpha\mapsto \dimh\big(\emoinsmualpha\big)$ and $\alpha\mapsto \dimp\big(\eplusmualpha\big)$ are called the \emph{lower} and the \emph{upper
singularity spectrum} of $\mu$. A measure for which $\dimh\big(\emoinsmualpha\big)$ or $\dimp\big(\eplusmualpha\big)$ is nonzero for several values of $\alpha$ is called \emph{multifractal}. 

\smallskip

In this paper, we are looking for measures $\mu$ having the worst possible behaviour: for any $\alpha\in[0,\dimh(K)]$, $\dimh\big(\emoinsmualpha\big)=\alpha$ or, for any $\alpha\in[0,\dimp(K)]$, $\dimp\big(\eplusmualpha\big)=\alpha.$
Can we find such a measure on any compact set $K$? Does a typical measure on $K$ satisfy this property? By a property true for a typical measure
of $\pk$, we mean a property which is satisfied by a dense $G_\delta$ set of elements of $\pk$. Throughout this paper, $\pk$ is endowed with the weak topology.

In this direction, a recent nice achievement was obtained by Z. Buczolich and S. Seuret; they show in \cite{BuSe10} that this is true if $K=[0,1]^d$.
\begin{BS}
 A typical measure $\mu\in\mathcal P([0,1]^d)$ satisfies, for any $\alpha\in[0,d]$, $\dimh\big(\emoinsmualpha\big)=\alpha$.
\end{BS}
Our first main result in this paper is the generalization of this last theorem to an arbitrary compact subset of $\mathbb R^d$.
\begin{theorem}\label{THMHAUSDORFF}
 Let $K$ be a compact subset of $\mathbb R^d$. A typical measure $\mu\in\pk$ satisfies, for any $\alpha\in\big[0,\dimh(K)\big)$, 
$\dimh\big(\emoinsmualpha\big)=\alpha$. If $\mathcal H^\alpha(K)>0$, this is also true for $\alpha=\dimh(K)$.
\end{theorem}

We now turn to the upper local dimension. In that case, the situation breaks down dramatically: for a typical measure $\mu\in\pk$, the sets $\Eplusmualpha$
are empty provided $\alpha$ is sufficiently small. This depends on the local upper box dimension of $K$. For $E$ a subset of $\mathbb R^d$, we denote by 
$\dboxsup(E)$ the upper-box dimension of $E$. The \emph{local upper box dimension} of a compact subset $K$ of $\mathbb R^d$ is defined by
$$\dboxsuploc(K)=\inf_{\substack{x\in K\\ \rho>0}}\dboxsup\big(K\cap B(x,\rho)\big).$$

\begin{theorem}\label{THMPACKING1}
 Let $K$ be a compact subset of $\mathbb R^d$. A typical measure $\mu\in\pk$ satisfies: for any $x\in K$, $\dsuplocmux\geq\dboxsuploc(K)$. 
\end{theorem}

This theorem has to be compared with a result of H. Haase. In \cite{Ha92}, he shows that if $x$ is a non-isolated point of $K$, then a typical measure
$\mu\in\pk$ satisfies $\dsuplocmux=+\infty$. Our Theorem \ref{THMPACKING1} is a kind of uniform version of Haase's result. We may also compare this result 
with a theorem of Genyuk, see \cite{Gen97}. She shows that, for a typical measure $\mu\in\pk$, $\dinflocmux=0$ and $\dsuplocmux=+\infty$
except on a set of first category. In particular, for a typical measure, the local dimension does typically not exist.

\smallskip

A Baire Category theorem shows that $\dboxsuploc(K)\leq\dimp(K)$ (see \cite[Lemma 10.18]{MP10}). The most interesting case of Theorem \ref{THMPACKING1}
happens when we have equality. This holds for many compact sets, like Ahlfors regular compact sets. Following G. David and S. Semmes \cite{DS97}, a
compact set $K\subset\mathbb R^d$ is called \emph{Ahlfors regular} if there are positive constants $s,c_1,c_2,r_0>0$ such that $\dimh(K)=s$
and 
$$c_1r^s\leq \mathcal H^s\big(K\cap B(x,r)\big)\leq c_2 r^s$$
for all $x\in K$ and all $0<r<r_0$, where $\mathcal H^s$ denotes the $s$-dimensional Hausdorff measure. Many compact sets are Ahlfors regular 
(see \cite{DS97}). For instance this is the case of self-similar compact sets satisfying the Open Set Condition. Ahlfors regularity easily implies that the local upper box dimension is equal to the packing dimension. Thus we get the following corollary.

\begin{corollary}\label{CORPACKING1}
 Let $K$ be a compact subset of $\mathbb R^d$ which is Ahlfors regular. Then a typical measure $\mu\in\pk$ satisfies: for any $x\in K$, $\dsuplocmux\geq\dimp(K)$. 
\end{corollary}
Since $\dsuplocmux\leq\dimp(K)$ for $\mu$-almost all $x\in K$ (and all $\mu\in\pk$), this corollary yields in turn:
\begin{corollary}
 Let $K$ be a compact subset of $\mathbb R^d$ which is Ahlfors regular. Then a typical measure $\mu\in\pk$ satisfies: for $\mu$-almost all $x\in K$, 
$\dsuplocmux=\dimp(K)$. 
\end{corollary}
In particular, if $K$ is Ahlfors regular, then a typical measure in $\pk$ is monodimensional from the upper local dimension point of view.

\medskip

However, one can ask whether there exists at least one measure with linear upper singularity spectrum. The answer is positive provided our
compact set contains an increasing family of subsets with prescribed packing dimension. When the packing and the Hausdorff dimension of these sets coincide,
we can go even further and exhibit a measure with a linear singularity spectrum.

\begin{theorem}\label{THMPACKING2}
 Let $K$ be a compact subset of $\mathbb R^d$ such that there exists an increasing family of compact subsets $(E_\alpha)_{0<\alpha<\dimp(K)}$ of $K$ such that, for any
$\alpha\in(0,\dimp(K))$, 
$$\dimp(E_\alpha)=\alpha\textrm{ and }\P^\alpha(E_\alpha)>0.$$
Then there exists $\mu\in\pk$ such that, for any $\alpha\in(0,\dimp(K))$, 
$$\dimp\big(\eplusmualpha\big)=\alpha.$$
Moreover, if we can choose the family $(E_\alpha)$ so that, for any
$\alpha\in(0,\dimp(K))$, 
$$\dimp(E_\alpha)=\alpha\textrm{ and }\H^\alpha(E_\alpha)>0,$$
then there exists $\mu\in\pk$ such that, for any $\alpha\in(0,\dimp(K))$, 
$$\dimh\big(\emualpha\big)=\dimp\big(\emualpha\big)=\alpha.$$
\end{theorem}

The assumptions of this last theorem are satisfied by many regular compact sets, like self-similar compact sets
satisfying the Open Set Condition.
\begin{corollary}\label{CORPACKING2}
 Let $K$ be a self-similar compact subset of $\mathbb R^d$ satisfying the Open Set Condition. 
Let $s=\dimh(K)=\dimp(K)$. Then there exists $\mu\in\pk$ such that, for any $\alpha\in(0,s)$, 
$$\dimh\big(\emualpha\big)=\dimp\big(\emualpha\big)=\alpha.$$
\end{corollary}

Of course, the same questions may be investigated for other forms of genericity. The notion of prevalence
(to be defined in Section 6)
 has been proved to be relevant in multifractal analysis (see for instance \cite{ABD07}, \cite{AMS13}, \cite{CN10}, \cite{FJ08},
\cite{FR13},  
\cite{Ol10}, \cite{Ol10b} or \cite{BAYHEUR2}) . The results that we get for this notion
are rather different: genericity of multifractal measures in the prevalence sense is easier to obtain.
 For instance, we shall prove the following result.
\begin{theorem}
Let $K$ be a self-similar compact subset of $\mathbb R^d$ satisfying the Open Set Condition. Then 
a prevalent measure $\mu\in\pk$ satisfies, for any $\alpha\in[0,\dimh(K))$, 
$$\dimh(\emualpha)=\dimp\big(\emualpha\big)=\alpha.$$
\end{theorem}

The paper is organized as follows. In Sections 2 and 3, we collect some preliminary results.
Section 4 is devoted to the study of the lower singularity spectrum whereas Section 5 investigates
the upper singularity spectrum. Finally in Section 6, we apply our previous results to the multifractal formalism
and to the notion of prevalence.

\section{Preliminaries}
\subsection{The topology on $\mathcal P(K)$}

Throughout this paper, $\mathcal P(K)$ will be endowed with the weak topology. It is well known (see for instance \cite{Par67})
that this topology is completely metrizable by the Fortet-Mourier distance defined as follows. Let $\textrm{Lip}(K)$ denote the family 
of Lipschitz functions $f:K\to\mathbb R$, with $|f|\leq 1$ and $\textrm{Lip}(f)\leq 1$, where $\textrm{Lip}(f)$ denotes
the Lipschitz constant of $f$. The metric $L$ is defined by
$$L(\mu,\nu)=\sup_{f\in\textrm{Lip}(K)}\left|\int fd\mu-\int fd\nu\right|$$
for any $\mu,\nu\in\mathcal P(K)$. We endow $\mathcal P(K)$ with the metric $L$. In particular, for $\mu\in\mathcal P(K)$
and $\delta>0$, $B_{L}(\mu,\delta)=\{\nu\in\mathcal P(K);\ L(\mu,\nu)<\delta\}$ will stand for the ball with center at $\mu$ and radius
equal to $\delta$. Since $K$ is separable, the set of probability measures with finite support $\fk$ is dense in $\pk$.

We shall use several times the following lemma. Its proof can be found e.g. in \cite{BAYLQ}.
\begin{lemma}\label{LEMTOPO1}
 For any $\gamma\in(0,1)$, for any $\beta>0$, there exists $\eta>0$ such that, for any $E$ a Borel subset of $K$, for any
$\mu,\nu\in\mathcal P(K)$, 
$$L(\mu,\nu)<\eta\implies \mu(E)\leq\nu\big(E(\gamma)\big)+\beta,$$
where $E(\gamma)=\{x\in K;\ \textrm{dist}(x,E)<\gamma\}$.
\end{lemma}

\subsection{Net measures}
To produce sets with prescribed Hausdorff dimension, it is convenient to use $s$-dimensional net measures $\Ms$.
Let $\mathcal N$ be the collection of all $d$-dimensional half-open binary cubes, specifically sets of the form
$$[2^{-k}m_1,2^{-k}(m_1+1))\times\dots\times [2^{-k}m_d,2^{-k}(m_d+1))$$
where $k$ is a nonnegative integer and $m_1,\dots,m_d$ are integers. If $E\subset\RR^d$ and $\delta>0$, define
$$\Ms_\delta(E)=\inf\sum_i |B_i|^s$$
where the infimum is over all countable $\delta$-covers of $E$ by sets $\{B_i\}$ of $\mathcal N$ (we
can always assume that the $B_i$ are disjoint). $\Ms_\delta$ is an outer measure on $\mathbb R^d$ and we let
$$\Ms(E)=\sup_{\delta>0}\Ms_\delta(E).$$
$\Ms$ and $\H^s$ are comparable measures. Precisely, there exists $B_d>0$ such that, for every $E\subset\RR^d$
and any $\delta\in(0,1)$, 
$$\H^s_\delta(E)\leq\Ms_\delta(E)\leq B_d\H^s_\delta(E).$$
We shall need the following very easy lemma:
\begin{lemma}\label{LEMNET1}
 Let $0<\alpha\leq\beta$, let $E\subset\mathbb R^d$ and let $\delta>0$. Then 
$$\M^\alpha_\delta(E)\geq\big(\M^\beta_\delta(E)\big)^{\beta/\alpha}.$$
\end{lemma}
\begin{proof}
 Let $(B_i)$ be a covering of $E$ by elements of $\mathcal N$ with diameters not exceeding $\delta$. Then
$$\M^\beta_\delta(E)\leq\sum_i |B_i|^\beta\leq\left(\sum_i |B_i|^\alpha\right)^{\alpha/\beta},$$
since $\|\cdot\|_{\beta/\alpha}\leq\|\cdot\|_1$. Taking the infimum over such coverings, we immediately get the result.
\end{proof}

\section{Sets with prescribed dimension}
\subsection{Sets with prescribed Hausdorff dimension.} To prove Theorem \ref{THMHAUSDORFF}, we need to produce a family $(E_\alpha)$ of subsets
of $K$ with $\dimh(E_\alpha)=\alpha$. For technical reasons, we will require that this family of sets is increasing, and that $\H^\alpha(E_\alpha)>0$. 
On $[0,1]^d$, this can be done using ubiquity arguments (for instance we can take the sets of real numbers $\beta$-approximable by dyadics for
$\beta$ describing $(1,+\infty)$). This was the starting 
point of \cite{BuSe10}. In our general context, the ubiquity argument will be replaced by the following theorem.
\begin{theorem}\label{THMBESICO}
 Let $K$ be a compact subset of $\RR^d$ and let $s=\dimh(K)$. There exists a family $(E_\alpha)_{\alpha\in(0,s)}$ of compact subsets
of $E$ satisfying 
\begin{itemize}
 \item $\forall\alpha\in(0,s)$, $0<\H^\alpha(E_\alpha)<+\infty$;
 \item $\forall 0<\alpha\leq \beta<s$, $E_\alpha\subset E_\beta$.
\end{itemize}
\end{theorem}
That for any $\alpha\in(0,s)$ there exists a compact subset $E_\alpha\subset E$ with $0<\H^\alpha(E_\alpha)<+\infty$ is known since Besicovitch.
Inspecting and modifying carefully the construction, we will ensure the second condition. 

For $a\in\mathbb R$, let $V_a$ be the affine subspace $\{x\in\mathbb R^d;\ x_1=a\}$. We denote by $\Gamma(K)$ the interval of the real numbers
$\alpha$ less than $s$ and satisfying 
\begin{eqnarray}\label{EQBESICO1}\H^\alpha(K\cap V_a)=0 \textrm{ for any }a\in\mathbb R.
\end{eqnarray}
 We shall first prove Theorem \ref{THMBESICO} with the restriction $\alpha\in\Gamma(K)$. 
This assumption is needed for the following lemma.
\begin{lemma}\label{LEMBESICO}
 Suppose that (\ref{EQBESICO1}) holds. Then for any compact set $L\subset K$, for any $k>0$,
the map $g:u\mapsto \M^\alpha_{2^{-k}}\big(L\cap\{x\in\mathbb R^d;\ x_1\leq u\}\big)$ is continuous.
\end{lemma}
\begin{proof}
 Fix $u_0\in\mathbb R$, $\veps>0$ and observe that $\M^\alpha_{2^{-k}}(L\cap V_{u_0})=0$. Thus we may find a finite
covering $(C_i)$ of $L\cap\{x_1=u_0\}$ by dyadic cubes of size less than $2^{-k}$ satisfying
$$\sum_i |C_i|^\alpha<\veps.$$
By compactness, these dyadic cubes also cover $L\cap\{u_0-\eta_1<x<u_0\}$ for some $\eta_1>0$. This yields easily, for
any $u\in[u_0-\eta_1,u_0)$, $g(u_0)-\veps \leq g(u)\leq g(u_0)$. On the other hand, let $(B_i)$ be a covering
of $L\cap\{x_1\leq u_0\}$ by dyadic cubes of size less than $2^{-k}$ satisfying
$$\sum_i |B_i|^\alpha<g(u_0)+\veps.$$
These cubes also cover $L\cap\{x\leq u_0+\eta_2\}$ for some $\eta_2>0$, so that $g(u_0)\leq g(u)\leq g(u_0)+\veps$
provided $u\in[u_0,u_0+\eta_2]$.
\end{proof}
We need to introduce the following notation.
For a dyadic cube $I=\prod_j [a_j,a_j+2^{-k})$ of size $2^{-k}$ and for $u\geq 0$, we shall denote by 
$$I_{|u}=I\cap\left([a_1,a_1+u]\times\prod_{j=2}^d [a_j,a_j+2^{-k})\right)=I\cap\{x\in\mathbb R^d;\ x_1\leq a_1+u\}.$$
$I_{|u}$ should be thought as the beginning of $I$. Observe that $I_{|0}$ is a face of the cube
and that $I_{|2^{-k}}=I$. Observe also that, if $J$ is a dyadic cube contained in $I$, then either $J\cap I_{|u}$
is empty or $J\cap I_u$ is equal to some $J_{|v}$.

\smallskip

We now prove Theorem \ref{THMBESICO} with $\alpha$ belonging only to $\Gamma(K)$. We first observe that, 
for any $\alpha\in\Gamma(K)$, $\M^\alpha_1(E)\in(0,+\infty)$. We shall build by induction on $k\geq 0$ 
compact sets $E_\alpha^k$, $\alpha\in\Gamma(K)$, satisfying the following properties:
\begin{itemize}
 \item[\bf (A)] For any $\alpha,\beta\in\Gamma(K)$ with $\alpha\leq\beta$, $E_\alpha^k\subset E_\beta^k$;
\item[\bf (B)] For any dyadic cube $I$ of size $2^{-(k-1)}$, for any $\alpha\in\Gamma(K)$, 
either one may find $u\geq 0$ such that $E_\alpha^k\cap I=K\cap I_{|u}$ or $E_\alpha^k\cap I=\varnothing$;
\item[\bf (C)] For any $\alpha\in\Gamma(K)$, $E_\alpha^{k+1}\subset E_\alpha^k$;
\item[\bf (D)] For any dyadic cube $I$ of size $2^{-(k-1)}$, for any $\alpha\in\Gamma(K)$, 
$$\M_{2^{-k}}^\alpha(E_\alpha^k\cap I)=\M_{2^{-(k-1)}}^\alpha(E_\alpha^{k-1}\cap I)\leq 2^{-\alpha(k-1)}.$$
\end{itemize}
Properties {\bf(A)} and {\bf(B)} are those which are new. Property {\bf(B)} implies that the same result holds automatically
for any dyadic cube $J$ contained in $I$.

\smallskip

We initialize the construction by setting $E_\alpha^0=K$, for any $\alpha\in\Gamma(K)$. Suppose that the sets $E_\alpha^k$
have been constructed and let us show how to construct $E_\alpha^{k+1}$. Let $\alpha\in\Gamma(K)$. We define $E_\alpha^{k+1}$
by specifying its intersection with a dyadic cube of size $2^{-k}$. Let $I$ be such a cube and let us consider the two following cases:
\begin{itemize}
 \item if $\M^\alpha_{2^{-(k+1)}}(E_\alpha^k\cap I)\leq 2^{-\alpha k}$, then we set $E_\alpha^{k+1}\cap I=E_\alpha^k\cap I$. 
In that case, properties {\bf(B)} and {\bf(D)} are easily satisfied. For {\bf(D)}, we shall observe that 
$\M^\alpha_{2^{-k}}(E_\alpha^k\cap I)=\M^\alpha_{2^{-(k+1)}}(E_k^\alpha\cap I)$ because using a dyadic cube of size $2^{-k}$ 
gives an estimate at least as large as $2^{-\alpha k}$.
\item if $\M^\alpha_{2^{-(k+1)}}(E_\alpha^k\cap I)>2^{-\alpha k}$, then we use the assumption $\alpha\in\Gamma(K)$. It implies that $\M^{\alpha}_{2^{-(k+1)}}(E_\alpha^k\cap I_{|0})=0$
and that the map $u\mapsto \M^{\alpha}_{2^{-(k+1)}}(E_\alpha^k\cap I_{|u})$ is continuous. Thus there exists a \emph{biggest} posive real number $u$ such that $\M^\alpha_{2^{-(k+1)}}(E_\alpha^k\cap I_{|u})=2^{-\alpha k}$. 
We set $E_\alpha^{k+1}\cap I=E_\alpha^k\cap I_{|u}$. Again, {\bf (B)} and {\bf (D)} are satisfied since $\M_{2^{-k}}^\alpha(E_\alpha^k\cap I)=2^{-\alpha k}.$
\end{itemize}
From the construction, {\bf(C)} is clear. It remains to prove {\bf(A)}. Let $\alpha,\beta\in \Gamma(K)$ with $\alpha\leq\beta$ and let $I$ be a dyadic cube of size $2^{-k}$.
On the one hand, if $E_\beta^{k+1}\cap I=E_\beta^k\cap I$, then the induction hypothesis ensures that $E_{\alpha}^{k+1}\cap I\subset E_{\beta}^{k+1}\cap I$.
On the other hand, suppose that $E_\beta^{k+1}\cap I$ is not equal to $E_\beta^k\cap I$. We may assume that $E_\alpha^{k+1}\cap I\neq\varnothing$
(otherwise there is nothing to prove). Then we write
$$E_\beta^{k+1}\cap I=K\cap I_{|u},\ E_\alpha^{k+1}\cap I=K\cap I_{|v},$$
where $u$ is the biggest real number with that property.
Suppose that $v>u$. Then, by maximality of $u$ and by Lemma \ref{LEMNET1},
$$\M_{2^{-(k+1)}}^\alpha(E_\alpha^{k+1}\cap I)\geq \left(\M_{2^{-(k+1)}}^\beta(K\cap I_{|v})\right)^{\alpha/\beta}>2^{\alpha k},$$
a contradiction with {\bf (D)}. Hence, $E_\alpha^{k+1}\subset E_\beta^{k+1}$ for every dyadic cube of size $2^{-k}$, showing {\bf (A)}.

\smallskip

We finally set $E_\alpha=\bigcap_k E_\alpha^k$. The arguments of \cite[Thm 5.4]{Fal85} show that $0<\H^\alpha(E_\alpha)<+\infty$. Moreover,
it follows from {\bf (A)} that $E_\alpha\subset E_\beta$
for any $\alpha\leq\beta$, $\alpha,\beta\in \Gamma(K)$.

\medskip

We now prove the full statement of Theorem \ref{THMBESICO} and we proceed by induction on $d$. The case $d=1$ has already been handled since (\ref{EQBESICO1})
is always true for any $\alpha>0$. To prove the statement for any value of $d$, let $t=\inf \Gamma(K)$ ($t=s$ provided $\Gamma(K)$ is empty, this can happen if $s\leq d-1$). We first assume that
$t\in\Gamma(K)$. We already know that there
exists a sequence $(F_\alpha)_{\alpha\in[t,s)}$ with $\H^\alpha(F_\alpha)\in(0,+\infty)$ and $F_\alpha\subset F_\beta$ provided $\alpha\leq\beta$. 

Let now $(\alpha_k)_{k\geq 1}$ be a sequence in $(0,t)$ increasing to $t$ and let $a_k\in\mathbb R$ be such that $\H^{\alpha_k}(K\cap V_{a_k})>0$
for any $k\geq 1$. By compactness of $K$, $(a_k)$ is bounded and we may assume that $(a_k)$ converges to some $a\in\RR$ and we set $\alpha_0=0$, $E_0=\varnothing$. The induction hypothesis
applied to $K\cap V_{a_k}$, which is contained in $\RR^{d-1}$,  gives us for every $k\geq 1$ an increasing sequence $(F_\alpha^k)_{0<\alpha\leq \alpha_k}$
of compact subsets of $K\cap V_{a_k}$ satisfying $\H^{\alpha}(F_\alpha^k)\in(0,+\infty)$ (we may include $\alpha=\alpha_k$
by looking first at a compact subset $K_k$ of $K\cap V_{a_k}$ with $\H^{\alpha_k}(K_k)\in(0,+\infty)$).

We define the sets $E_\alpha$, for $0<\alpha<t$, in the following way. Let $k$ be such that $\alpha$ belongs to $(\alpha_{k-1},\alpha_k]$.
We set (inductively) $E_\alpha=F_\alpha^k\cup E_{\alpha_{k-1}}$. It is worth noting that 
$$E_\alpha\subset K\cap \left(\bigcup_k V_{a_k}\cup V_a\right)=:W,$$
which is a compact subset of $K$. The sets $F_\alpha$, for $\alpha\geq t$, do not necessarily contain $E_\beta$, for $\beta<t$. This leads us to set,
for $\alpha\geq t$, $E_\alpha=F_\alpha\cup W$. This does not affect their $\alpha$-dimensional Hausdorff measure, since
$$\H^t(W)\leq\sum_k \H^t(K\cap V_{a_k})+\H^t(K\cap V_a)=0.$$

We finally assume that $t\notin \Gamma(K)$, so that $\H^t(K\cap V_a)>0$ for some $a\in\RR$. We apply the induction hypothesis to $K\cap V_a$ to obtain an increasing family $(E_\alpha)_{0<\alpha\leq t}$ of compact
sets satisfying $0<\H^\alpha(E_\alpha)<+\infty$ for any $\alpha\in(0,t]$. We also know how to construct a similar family $(F_\alpha)_{t<\alpha<s}$. We just modify slightly this last family
by setting $E_\alpha=F_\alpha\cup E_t$ for $\alpha>t$. The family $(E_\alpha)_{0<\alpha<s}$ satisfies all the requirements of Theorem \ref{THMBESICO}.

\begin{remark}
 The proof of Theorem \ref{THMBESICO} is really specific to the net structure. We do not know whether a similar result does exist on general metric spaces.
\end{remark}

\subsection{Sets with prescribed packing dimension}
Let $K$ be a compact set with packing dimension $s$. A result of H. Joyce and D. Preiss (\cite{JP95}) ensures that, for any $\alpha\in(0,s)$,
there exists $E_\alpha\subset K$ compact with $0<\P^\alpha(E_\alpha)<+\infty$. We would like to find an increasing family $(E_\alpha)_{0<\alpha<s}$.
We do not know whether this is possible in full generality. However, let us show how to construct such a family in the context of self-similar compact sets.

We recall that a compact subset $K$ of $\RR^d$ is called \emph{self-similar} if there are finitely many contractive similarity maps $S_1,\dots,S_m:\RR^d\to\RR^d$
such that $K=\bigcup_i S_i(K)$. We denote by $r_1,\dots,r_m$ the ratios of these maps. $K$ is said to satisfy the \emph{Open Set Condition} if there is an open, non-empty and bounded subset $\mathcal O$ of $\mathbb R^d$
such that $S_i(\mathcal O)\subset \mathcal O$ for all $i$ and $S_i(\mathcal O)\cap S_j(\mathcal O)=\varnothing$ for all $i\neq j$.

\begin{proposition}
 Let $K$ be a self-similar compact subset of $\RR^d$ satisfying the Open Set Condition. Let $s=\dimp(K)=\dimh(K)$. There exists a family $(E_\alpha)_{0<\alpha<s}$ 
of compact subsets of $\RR^d$ satisfying
$$\forall \alpha\in(0,s),\ \H^\alpha(E_\alpha)>0\textrm{ and }\dimp(E_\alpha)\leq\alpha.$$
\end{proposition}
\begin{proof}
 We shall use two results regarding Besicovitch subsets of self-similar compact sets. We need to introduce notations. Let $\Sigma=\{1,\dots,m\}^\NN$ be the symbolic space
and let $\Sigma^*=\bigcup_n\{1,\dots,m\}^n$ be the family of all finite lists with entries in $\{1,\dots,m\}$. For $\bfi \in\Sigma$ and a positive
integer $n$, let $\bfi |n$ denote the truncation of $\bfi $ to the $n$-th place. Finally, for $\bfi=i_1\dots i_n\in\Sigma^*$, we write $S_{\bfi }=S_{i_1}\circ\dots\circ S_{i_n}$
and $K_{\bfi }=S_{\bfi }K$. We define a map $\pi:\Sigma\to K$ by $\{\pi(\bfi )\}=\bigcap_n K_{\bfi|n}$.

We then define $m$ maps $\Pi_i:\Sigma^*\to\RR$ by 
$$\Pi_i(\bfi)=\frac{|\{1\leq j\leq n;\ i_j=i\}|}n,$$
namely $\Pi_i(\bfi)$ gives the frequency of the digit $i$ in $\bfi$. Let $\Pi=(\Pi_1,\dots,\Pi_m)$ be the vector of the frequencies of the digits.
For $\bfi\in\Sigma$, let $A(\bfi)$ denote the set of accumulation points of $(\Pi(\bfi|n))_n$. $A(\bfi)$ is contained in $\Delta$, the simplex of the probability
vectors in $\RR^m$, namely $\Delta=\{\bfp=(p_1,\dots,p_m)\in\RR^m;\ p_j\geq 0\textrm{ and }\sum_j p_j=1\}$.

One can control the dimension of the set of points $\Pi(\bfi)$, $i\in\Sigma$, such that $A(\bfi)$ belongs to a fixed subset of $\Delta$. This involves the notion of entropy.
Define $\Lambda:\Delta\to \RR$ by 
$$\Lambda(\bfp)=\frac{\sum_j p_j\log p_j}{\sum_j p_j\log r_j}.$$
In particular, since $s$ satisfies $r_1^s+\dots+r_m^s=1$, 
$$\Lambda(r_1^s,\dots,r_m^s)=s.$$
For $C$ a subset of $\Delta$, we define
$$K(C)=\pi\big(\{\bfi\in\Sigma;\ A(\bfi)\subset C\}\big)\subset K.$$
The first result that we need is due to J-H. Ma, Z-Y. Wen and J. Wu in \cite{MWW02}.
$$\textrm{For any }\bfp\in\Delta,\ \bfp\neq(r_1^s,\dots,r_m^s),\ \H^{\Lambda(\bfp)}\big(K(\{p_1,\dots,p_m\})\big)>0.$$
The second result that we need is due to L. Olsen and S. Winter in \cite{OW03}.
$$\textrm{For any compact and convex set }C\subset\Delta,\ \dimp\big(K(C)\big)=\sup_{\bfp\in C}\Lambda(\bfp).$$

\medskip

For any $\lambda\in[0,1]$, we define
\begin{eqnarray*}
 C(\lambda)&=&\{\bfp\in\Delta;\ 0\leq p_j\leq\lambda r_j^s\textrm{ for any }j=1,\dots,m-1\}\\
f(\lambda)&=&\max_{\bfp\in C(\lambda)}\Lambda(\bfp).
\end{eqnarray*}
$f$ is continuous, nondecreasing and satisfies $f(0)=0$, $f(1)=s$. For any $\alpha\in(0,s)$, we define $g(\alpha)=\sup\{\lambda\in[0,1];\ f(\lambda)=\alpha\}$.
$g$ is a right-inverse of $f$ and it is increasing. Finally, let us set $E_\alpha=K\big(C(g(\alpha))\big)$. The family $(E_\alpha)_{0<\alpha<s}$ 
is increasing and $\dimp(E_\alpha)=f\circ g(\alpha)=\alpha$. Moreover, let $\bfp\in C(g(\alpha))$ be such that $\Lambda(\bfp)=\alpha$. Then 
$$\H^\alpha\big(E_\alpha\big)\geq \H^\alpha\big(K(\{p_{1},\dots,p_m\})\big)>0.$$
\end{proof}

\section{Multifractal measures - the Hausdorff case}
\subsection{At a fixed level.}
To construct multifractal measures on $K$, we begin by fixing some compact subset $E$ of $K$  with $\dimh E=\alpha$
and by building plenty of measures $\mu\in\pk$ so
that $E\subset \Emoinsmualpha$. This is reminiscent from Lemma 3.5 of \cite{Cut95} but our result is more precise
(and the method of proof, which is inspired by \cite{BuSe10}, is different). Without loss of generality, we may assume that $K$ is infinite.
\begin{proposition}\label{PROPHAUSDORFF}
Let $E$ be a compact subset of $K$ and let $\alpha>0$. Assume that $\H^\alpha(E)<+\infty$. Then there exists a dense $G_\delta$-subset $\R_E$ in $\pk$
such that, for any $\mu\in\R_E$, $E\subset\Emoinsmualpha$.
\end{proposition}
\begin{proof}
We first assume that $\H^\alpha(E)=0$. Then for any $n\geq 1$, we can find a set $\mathcal B_n$ of balls with diameters less than $2^{-n}$, 
covering $E$, and such that
$$\sum_{B\in\mathcal B_n}|B|^\alpha\leq 2^{-(n+1)}.$$
We may assume that $\mathcal B_n$ is finite and that every ball in $\mathcal B_n$ does intersect $E$. Let $\rho_n$ be the minimum of the diameters of the balls in $\mathcal B_n$. We set $\mathcal C_n=\bigcup_{B\in\mathcal B_n}B$.
 Let also $(\omega_n)$ be an increasing sequence of positive real numbers going to $+\infty$ and such that 
 $$\sum_{n\geq 1}\omega_n\sum_{B\in\mathcal B_n}|B|^\alpha=1.$$
For any $n\geq 1$ and any ball $B\in\mathcal B_n$, we pick a point $x_B\in K\cap B$. We then set 
 $$\mu_0=\sum_{n\geq 1}\omega_n\sum_{B\in\mathcal B_n}|B|^\alpha \delta_{x_B}.$$
 Let now $\nu\in\fk$ and let $N$ be the number of elements in the support of
 $\nu$. We define
 $$\mu_\nu=\left(1-\frac1N\right)\nu+\frac 1N\mu_0.$$
 Since $K$ is infinite, $\{\mu_\nu\}_{\nu\in\fk}$ is dense in $\pk$. Let now $l\geq 1$ and let $n\geq l$. For any $y\in\mathcal C_n$,
 we may find a ball $B\in\mathcal B_n$ with radius $r\in[\rho_n,2^{-n}]$ so that $y\in B$. This in turn implies $x_B\in B(y,2r)$. Thus
 \begin{eqnarray*}
 \mu_\nu\big(B(y,2r)\big)&\geq&\frac 1N\mu_0\big(B(y,2r)\big)\\
 &\geq&\frac{\omega_n}Nr^\alpha\\
 &\geq&8^\alpha r^\alpha,
 \end{eqnarray*}
 provided $n$ is large enough. We shall denote this integer by $n_{\nu,l}$. By Lemma \ref{LEMTOPO1} applied with the parameters
$\gamma=2r$ and $\beta=4^\alpha\rho_{n_{\nu,l}}$, one can find $\eta_{\nu,l}>0$ such that 
 any $\mu\in\pk$ with $L(\mu,\mu_\nu)<\eta_{\nu,l}$ satisfies, for any $y\in \mathcal C_{n_{\nu,l}}$ and for the associated $r>0$,
 $$\mu\big(B(y,4r)\big)\geq \mu_{\nu}\big(B(y,2r)\big)-4^\alpha\rho_{n_{\nu,l}}\geq 4^\alpha r^\alpha.$$
 We finally consider the dense $G_\delta$-set 
 $$\R_E=\bigcap_{l\geq 1}\bigcup_{\mu_\nu\in\fk}B_L(\nu,\eta_{\nu,l}).$$
 Pick $\mu\in\R_E$. For any $l\geq 1$, there exists $\nu\in\fk$ such that $\mu\in B_L(\mu_\nu,\eta_{\nu,l})$. Any $y\in E$ belongs
 to $\mathcal C_{n_{\nu,l}}$ so that there exists $0<r\leq 2^{-n_{\nu,l}}\leq 2^{-l}$ with
 $$\mu\big(B(y,4r)\big)\geq 4^\alpha r^\alpha.$$ 
 This means that $E\subset\Emoinsmualpha$.
 
 \medskip
 
 Suppose now that we have the weaker assumption $\H^\alpha(E)<+\infty$. For any $\beta>\alpha$, $\H^\beta(E)=0$ and the first part tells us that
 we can find a residual subset $\R_\beta\subset\pk$ such that, for any $\mu\in\R_\beta$, $\mathcal E_-(\mu;\beta)\supset E$. Let $(\beta_k)$
 be a decreasing sequence of positive real numbers going to $\alpha$ and let $\R=\bigcap_k \R_{\beta_k}$. $\R$ keeps being a residual subset
 of $\pk$ and, taking the limit, $\Emoinsmualpha\supset E$.
\end{proof}

\subsection{Multifractal measures}
We now prove Theorem \ref{THMHAUSDORFF}. We start with $K$ a compact subset of $\mathbb R^d$ with $\dimh (K)=s$. Let $(E_\alpha)_{0<\alpha<s}$
be the family of sets given by Theorem \ref{THMBESICO}. We also pick $(\alpha_k)$ a dense sequence in $(0,s)$. For each $k$, Proposition \ref{PROPHAUSDORFF}
gives us a residual subset $\mathcal R_k$  such that 
$$\mu\in\mathcal R_k\implies E_{\alpha_k}\subset \E_-(\mu;\alpha_k).$$
We set $\mathcal R=\bigcap_k\mathcal R_k$ and we claim that $\mathcal R$ is the residual subset we are looking for. Indeed, let $\mu\in\R$ and let
$\alpha\in(0,s)$. We consider a subsequence $(\alpha_{\varphi(k)})$ decreasing to $\alpha$. Since, for any $k$,
$$E_\alpha\subset E_{\alpha_{\varphi(k)}}\subset \mathcal E_-(\mu;\alpha_{\varphi(k)}),$$
we get by taking the limit $E_\alpha\subset \Emoinsmualpha$. We then split $E_\alpha$ into two parts:
\begin{eqnarray*}
E_\alpha^1&=&\big\{x\in E_\alpha;\ \dinflocmux=\alpha\big\}\\
E_\alpha^2&=&\big\{x\in E_\alpha;\ \dinflocmux<\alpha\big\}.
\end{eqnarray*}
The crucial point is $\H^\alpha(E_\alpha^2)=0$. Indeed, $E_\alpha^2$ is contained in a countable union of sets $\E_-(\mu;\beta)$, $\beta<\alpha$, and these
sets satisfy $\H^\alpha\big(\E_-(\mu;\beta)\big)=0$. On the contrary, 
$\H^\alpha(E_\alpha)>0$. We then conclude that $\H^\alpha(E_\alpha^1)>0$ and $E_\alpha^1\subset\emoinsmualpha$. Hence, this last set has Hausdorff dimension equal to $\alpha$.

\medskip

When $\H^s(K)>0$, we can include $\alpha=\dimh(K)$ in the statement. Indeed, we can always assume that $0<\H^s(K)<+\infty$ and then we just set $E_s=K$. 
The rest of the proof is unchanged.

\section{Multifractal measures - the packing case}

\subsection{Obstructions to residuality of multifractal measures}
In this subsection, we prove Theorem \ref{THMPACKING1}. We first observe that it is enough to prove that, for any $s\in(0,\dboxsuploc(K))$, a typical measure
$\mu\in\pk$ satisfies $\dsuplocmux\geq s$ for any $x\in K$. Taking a sequence $(s_k)$ increasing to $\dboxsuploc(K)$ and the countable intersection
of residual sets, we will find the desired statement. Thus, let us fix $s\in(0,\dboxsuploc(K))$ and let $m>l\geq 1$. We define
$$\Ulm=\left\{\mu\in\pk;\ \forall x\in K,\ \exists r\in(1/m,1/l)\ \textrm{s.t.}\ \mu\big(B(x,r)\big)< r^s\right\}.$$
It is easy to check that any $\mu$ in $\bigcap_{l\geq 1}\bigcup_{m>l}\Ulm$ satisfies $\dsuplocmux\geq s$ for any $x\in K$. Hence, we just need
to prove the two following facts.
\begin{description}
\item[Fact 1] Each $\Ulm$ is open.
\item[Fact 2] For any $l\geq 1$, $\bigcup_{m> l}\Ulm$ is dense.
\end{description}
To prove Fact 1, we pick $\mu\in\Ulm$. Then, for any $x\in K$, we may find $r_x\in(1/m,1/l)$ such that 
$\mu\big(B(x,r_x)\big)<r_x^s$. Let $\veps_x\in(0,1)$ be small enough so that 
$$\left\{
\begin{array}{rcl}
(1-2\veps_x)r_x&>&1/m\\
\mu\big(B(x,r_x)\big)&<&(1-3\veps_x)^sr_x^s.
\end{array}\right.$$
We set $\eta_x=\veps_xr_x$. For any $y\in B(x,\eta_x)$, the ball $B\big(y,(1-\veps_x)r_x\big)$ is contained in $B(x,r_x)$ so that
$\mu\big(B(y,(1-\veps_x)r_x)\big)<(1-3\veps_x)r_x^s$. The compact set $K$ may be covered by a finite number of balls $B(x_i,\eta_i)$, $i=1,\dots,m$.
Let $\eta>0$ be given by Lemma \ref{LEMTOPO1} with $\gamma=\inf_i\veps_ir_i$ and 
$\beta=\inf_i\big((1-2\veps_i)^sr_i^s-(1-3\veps_i)^sr_i^s\big)$. Let $\nu\in\pk$ with $L(\mu,\nu)<\eta$ and let $y\in K$. Then $y$ belongs to some
$B(x_i,\eta_i)$ and
$$\nu\big(B(y,(1-2\veps_i)r_i)\big)\leq \mu\big(B(y,(1-\veps_i)r_i)\big)+\beta\leq (1-2\veps_i)^s r_i^s.$$
This shows that $\nu$ belongs to $\Ulm$ and that $\Ulm$ is open.

\smallskip

Let us now prove Fact 2. Let $l\geq 1$, $\mu=\sum_{i=1}^n p_i\delta_{x_i}\in\fk$ and $\eta>0$. We need to find $\nu\in\bigcup_{m>l}\Ulm$
satisfying $L(\mu,\nu)<\eta$. Let $\rho>0$ be such that
$$\rho<4\min(\|x_i-x_j\|;\ i\neq j)\textrm{ and } \rho<\eta.$$
For each $i\in\{1,\dots,n\}$, $\dboxsup\big(K\cap B(x_i,\rho)\big)>s$. This implies that we may find $N_i$ as large as we want and points $x_{i,1},\dots,x_{i,N_i}$
in $K\cap B(x_i,\rho)$ which are the centers of disjoint balls of radius $4N_i^{-s}$. We assume that the integers $N_i$ are sufficiently large so that 
the balls $B(x_i,\rho+4N_i^{-s})$ are pairwise disjoint and we also assume that $2\sup_iN_i^{-s}<1/l$. We set 
$$\mu_i=\frac{1}{N_i}\sum_{j=1}^{N_i}\delta_{x_{i,j}}\textrm{ and }\nu=\sum_{i=1}^n p_i\mu_i.$$
We claim that $\nu$ belongs to $\bigcup_{m>l}\Ulm\cap B_L(\mu,\eta)$. Indeed, for any $f\in\textrm{Lip}(K)$, 
$$\left|\int fd\mu-\int fd\nu\right|\leq \sum_i p_i \int |f-f(x_i)|d\mu_i\leq \rho<\eta.$$
On the other hand, let $m>l$ be such that $\frac1m<\inf_i N_i^{-s}$ and let $x\in K$. 
\begin{itemize}
\item either $x$ belongs to some $B(x_i,\rho+2N_i^{-s})$. We set $r=2N_i^{-s}\in (1/m,1/l)$. At most one point $x_{i,j}$, $j=1,\dots,N_i$, may belong to $B(x,r)$ since
$\|x_{i,j}-x_{i,l}\|>2r$ for $j\neq l$. Hence,
$$\nu\big(B(x,r)\big)\leq\frac 1{N_i}<r^s.$$
\item or $x$ does not belong to any $B(x_i,\rho+2N_i^{-s})$. In that case, we let $r=2\inf_i N_i^{-s}$ and we observe that
$$\nu\big(B(x,r)\big)=0<r^s.$$ 
\end{itemize}
This yields $\nu\in\Ulm$. Hence $\bigcup_{m>l}\Ulm$ is dense in $\pk$, which achieves the proof of Theorem \ref{THMPACKING1}.

\subsection{Existence of one multifractal measure}
We now investigate the existence of at least one measure with linear upper singularity spectrum. Our starting point is an analogue of Proposition \ref{PROPHAUSDORFF}.
\begin{lemma}\label{LEMPACKING}
 Let $E$ be a compact subset of $K$ and let $\alpha>0$. Assume that $\P^\alpha(E)<+\infty$. Then there exists $\mu\in\pk$
such that $E\subset\Eplusmualpha$.
\end{lemma}
\begin{proof}
 Let $(\alpha_k)$ be a sequence of positive real numbers decreasing to $\alpha$, so that $\dimp(E)<\alpha_k$. A result of C. Tricot (\cite{Tri82})
ensures that there exists $\mu_k\in\pk$ such that, for any $x\in E$, $\dsuplocmux\leq\alpha_k$. 
We now define $\mu=\sum_k 2^{-k}\mu_k$ and we observe that, for any $x\in E$ and any $k\geq 1$, 
$$\dsuplocmux\leq\overline{\dim}_{\rm loc}(\mu_k;x)\leq \alpha_k.$$
Taking the limit, we get $E\subset\mathcal E^+(\mu;\alpha)$.
\end{proof}
\begin{proof}[Proof of Theorem \ref{THMPACKING2}]
 We start with an increasing family of sets $(E_\alpha)$, $0<\alpha<\dimp(K)$, with $\dimp(K)=\alpha$
and $\P^\alpha(E_\alpha)>0$. Let $(\alpha_k)$ be a sequence of positive real numbers dense in $(0,\dimp(K))$ and let $\mu_k$ be given 
by Lemma \ref{LEMPACKING} for $\alpha=\alpha_k$ and $E=E_{\alpha_k}$. We set $\mu=\sum_k 2^{-k}\mu_k$ and we
claim that $\dimp\big(\Eplusmualpha\big)=\alpha$ for any $\alpha\in\big(0,\dimp(K)\big)$.

Indeed, let us fix $\alpha\in\big(0,\dimp(K)\big)$ and let $(\alpha_{\phi(k)})$ be a subsequence of $(\alpha_k)$
decreasing to $\alpha$. Since $E\subset\bigcap_k E_{\alpha_{\phi(k)}}$ and since 
$$\dsuplocmux\leq\dsuploc(\mu_{\phi(k)};x)\leq\alpha_{\phi(k)}$$
for any $x\in E_{\alpha_{\phi(k)}}$, one gets $E_\alpha\subset \Eplusmualpha$. We then argue like in the proof
of Theorem \ref{THMHAUSDORFF}, splitting $E_\alpha$ into two parts:
\begin{eqnarray*}
 E_\alpha^1&=&\big\{x\in E_\alpha;\ \dsuplocmux=\alpha\big\}\\
E_\alpha^2&=&\big\{x\in E_\alpha;\ \dsuplocmux<\alpha\big\}.
\end{eqnarray*}
Since $\P^\alpha(E_\alpha^2)=0$ and $\P^\alpha(E_\alpha)>0$, one gets $\P^\alpha(E_\alpha^1)>0$, exactly
what is needed to conclude.

\medskip

Assume now that our sequence of sets $(E_\alpha)$ satisfies not only $\P^\alpha(E_\alpha)>0$, but also
$\H^\alpha(E_\alpha)>0$. In that case, we split $E_\alpha$ into
\begin{eqnarray*}
 F_\alpha^1&=&\big\{x\in E_\alpha;\ \dinflocmux\geq \alpha\big\}\\
F_\alpha^2&=&\big\{x\in E_\alpha;\ \dinflocmux<\alpha\big\}.
\end{eqnarray*}
Then $\H^\alpha(F_\alpha^1)>0$ and $\dimp(F_\alpha^1)\leq\dimp(E_\alpha)\leq \alpha$, showing
that $\dimh(F_\alpha^1)=\dimp(F_\alpha^1)=\alpha$. Moreover any $x\in F_\alpha^1$ satisfies
$$\alpha\leq\dinflocmux\leq\dsuplocmux\leq\alpha$$
which yields the second part of Theorem \ref{THMPACKING2}.

\end{proof}

\section{Applications}
\subsection{Multifractal formalism}
Let $\mu\in\pk$. The lower $L^q$-spectrum of $\mu$ is defined by 
$$\dinfmuq=\liminf_{r\to 0}\frac{\log \int_K \big(\mu(B(x,r))\big)^{q-1}d\mu(x)}{\log r}.$$
It is well-known (see \cite{BMP92,Ol95}) that the Legendre transform of $\mu$ is an upper bound for the multifractal spectrum of $\mu$: for any $\alpha\geq 0$, 
$$\dimh\big(\emoinsmualpha\big)\leq (\dinfmu)^*(\alpha)=:\inf_{q\in\RR}\big(q\alpha-\dinfmuq\big).$$
For many specific measures (like self-similar measures), this upper bound turns out to be an equality. We say that the measure satisfies the multifractal formalism at $\alpha$
if the previous equality holds for $\alpha$. The validity of the multifractal formalism is a very important issue in Physics: when it is known to be satisfied, we can estimate the singularity spectrum of real data
 through the Legendre spectrum.

\medskip

We now show that, in any compact subset of $\RR^d$, a typical measure satisfies the multifractal formalism. This extends the result of Buczolich and Seuret for
typical measures on the cube.
\begin{theorem}\label{THMFORMALISM}
Let $K$ be a compact subset of $\RR^d$. A typical measure $\mu\in\pk$ satisfies
$$\dimh\big(\emoinsmualpha\big)=(\dinfmu)^*(\alpha)\quad\quad\textrm{for any }\alpha\in[0,\dimh(K)).$$
\end{theorem}
\begin{proof}
By \cite{BAYLQ} and \cite{Ol05}, we know that a typical measure $\mu\in\pk$ satisfies
$$\dinfmuq=\begin{cases}
0&\textrm{if }q\geq 1\\
-s(1-q)&\textrm{if }q\in[0,1)\\
-\infty&\textrm{if }q<0,
\end{cases}$$
where $s$ is the upper-box dimension of $K$. Then it is easy to check that, for any $\alpha\in[0,s]$, $(\dinfmu)^*(\alpha)=\alpha$. 
We conclude by comparing this with the statement of Theorem \ref{THMHAUSDORFF} (recall that $\dimh(K)\leq s)$. 
\end{proof}

\subsection{Prevalence}
In this paper, we are mainly concerned with genericity in the sense of Baire theorem. We can also consider other notions of genericity.
In particular, genericity in the sense of prevalence has been proved to be relevant in multifractal analysis. Prevalence is an infinite-dimensional version of ``almost-everywhere". It has been defined by Anderson and Zame \cite{AZ01} in the context of convex subsets of vector spaces (see also \cite{OY05}).

Let $X$ be a topological vector space and let $C\subset X$ be a convex subset of $X$ which is completely metrizable in the relative topology induced from $X$. 
Fix a Borel set $E\subset C$ and $c\in C$.  We say that $E$ is \emph{shy} in $C$ at $c$ if, for any $\eta>0$ and any neighbourhood $U$ of $0$, 
there is a Borel probability measure $\Lambda$ on $X$ such that
\begin{enumerate}
\item $\textrm{supp}(\Lambda)$ is compact and $\supp(\Lambda)\subset U+c$;
\item $\supp(\Lambda)\subset \eta C+(1-\eta)c$;
\item for all $x\in E$, we have $\Lambda(x+E)=0$.
\end{enumerate}
A Borel subset of $C$ is said \emph{prevalent} if its complement (in $C$) is shy in $C$ at some point $c\in C$.

\medskip

The behaviour of the multifractal spectrum for a prevalent measure was first studied by Olsen in \cite{Ol10b}. He shows that, for any $\alpha>0$, a prevalent
measure $\mu\in\pk$ satisfies
\begin{eqnarray*}
\dimh\big(\Emoinsmualpha\big)&=&\sup_{\nu\in\pk}\dimh\big(\E_-(\nu;\alpha)\big)\\
\dimp\big(\Eplusmualpha\big)&=&\sup_{\nu\in\pk}\dimp\big(\E_+(\nu;\alpha)\big).
\end{eqnarray*}
When $K$ is a self-similar compact set satisfying the Open Set Condition, he concludes that for any $\alpha\in[0,\dimh(K)]$, a prevalent measure
$\mu\in\pk$ satisfies $\dimh\big(\Emoinsmualpha\big)=\dimp\big(\Eplusmualpha\big)=\alpha$. 

\smallskip

An application of Proposition \ref{PROPHAUSDORFF} shows that the assumption ``$K$ is a self-similar compact set satisfying the Open
Set Condition" is not needed at least to prove that, for any $\alpha\in[0,\dimh(K))$, a prevalent measure $\mu\in\pk$
satisfies $\dimh\big(\Emoinsmualpha\big)=\alpha$. We shall go much further by showing that a prevalent measure has a linear lower singularity spectrum.
When the compact set is regular, we can strenghten this statement: a prevalent measure has a linear singularity spectrum. Hence genericity in the Baire sense
and in the prevalence sense give very different results.
\begin{theorem}
Let $K$ be a compact subset of $\mathbb R^d$. Then a prevalent measure $\mu\in\pk$ satisfies, for any $\alpha\in[0,\dimh(K))$, 
$$\dimh\big(\emoinsmualpha\big)=\alpha.$$
If $K$ is a self-similar compact set satisfying the Open Set Condition, then 
a prevalent measure $\mu\in\pk$ satisfies, for any $\alpha\in[0,\dimh(K))$, 
$$\dimh(\emualpha)=\dimp\big(\emualpha\big)=\alpha.$$
\end{theorem}
\begin{proof}
We shall see that the existence of one measure with a linear multifractal spectrum implies that a prevalent measure shares the same property. 
We just prove the first statement, the second one being completely similar. We fix $a\in K$ and let 
$$M=\left\{\mu\in\pk;\ \exists \alpha\in\big(0,\dimh(K)\big), \dimh\big(\emoinsmualpha\big)<\alpha\right\}.$$
We show that $M$ is a shy in $\pk$ at $\delta_a$. Namely, given $\eta>0$ and $U$ a neighbourhood of $0$ in $\M(K)$, we have to construct a Borel
probability measure $\Lambda$ on $\mathcal M(K)$ satisfying the following conditions:
\begin{enumerate}
\item $\supp(\Lambda)$ is compact and $\supp(\Lambda)\subset U+\delta_a$;
\item $\supp(\Lambda)\subset \eta\pk+(1-\eta)\delta_a$;
\item for all $\mu\in\mathcal M(K)$, $\Lambda(\mu+M)=0$.
\end{enumerate}
Let $\Theta$ be a multifractal probability measure on $K$ given by Theorem \ref{THMHAUSDORFF}. Moreover, we ask that for any
$\alpha\in(0,\dimh(K))$, $\H^\alpha\big(E(\Theta;\alpha)\big)>0$ (see the proof of this theorem). Let $\rho>0$ be very small and let
$T:[0,\rho]\to \pk$ be defined by
$$T(t)=t\Theta+(1-t)\delta_a.$$
The probability measure $\Lambda$ on $\mathcal M(K)$ is defined by
$$\Lambda=\frac1\rho\mathcal L\circ T^{-1},$$
where $\mathcal L$ is the $1$-dimensional Lebesgue measure on $\mathbb R$. Provided $\rho$ is sufficiently small, (1) and (2) are clearly satisfied.
We now show (3) and even a much stronger property: for any $\mu\in\M(K)$, there exists at most one $t\in[0,\rho]$ such that 
$T(t)\in\mu+M$. Suppose that the contrary holds, and let $t<s$, $\nu_1,\nu_2\in M$ be such that
$$\mu+\nu_1=t\Theta+(1-t)\delta_a,\quad \mu+\nu_2=s\Theta+(1-s)\delta_a.$$
Substracting these equalities, we find
$$(s-t)\Theta=\nu_2-\nu_1+(s-t)\delta_a\leq \nu_2+(s-t)\delta_a.$$
This implies that, for any $x\in K\backslash\{a\}$, 
$$\dinfloc(\nu_2;x)\geq \dinfloc(\Theta;x).$$
Thus, for any $\alpha\in\big(0,\dimh(K)\big)$, 
$$E_-(\Theta;\alpha)\subset \E_-(\nu_2;\alpha)\cup\{a\}.$$
This in turn yields
$$\H^\alpha\big(\E_-(\nu_2;\alpha)\big)>0.$$
We argue like in the proof of Theorem \ref{THMHAUSDORFF} to conclude that $\H^\alpha\big(E_-(\nu_2;\alpha)\big)>0$ for any
$\alpha\in\big(0,\dimh(K)\big)$, a contradiction with $\nu_2\in M$.
\end{proof}


\begin{thebibliography}{MWW02}

\bibitem[ABD07]{ABD07}
J-M. Aubry, F.~Bastin, and S.~Dispa, \emph{{Prevalence of multifractal
  functions in $S_\nu$-spaces}}, J. Fourier Anal. Appl. \textbf{13} (2007),
  175--185.

\bibitem[AMS]{AMS13}
J-M. Aubry, D.~Maman, and S.~Seuret, \emph{{Local behavior of traces of Besov
  functions: Prevalent results}}, preprint.

\bibitem[AZ01]{AZ01}
R.~Anderson and R.~Zame, \emph{Genericity with infinitely many parameters},
  Adv. Theor. Econ., vol.~1, 2001.

\bibitem[Bay12]{BAYLQ}
F.~Bayart, \emph{{How behave the typical $L^q$-dimensions of measures?}},
  preprint (2012), arXiv:1203.2813.

\bibitem[BH11]{BAYHEUR2}
F.~Bayart and Y.~Heurteaux, \emph{{Multifractal analysis of the divergence of
  Fourier series, the extreme cases}}, preprint (2011), arXiv:1110.5478.

\bibitem[Bil60]{Bil60}
P.~Bilingsley, \emph{Hausdorff dimension in probability theory}, Illinois J.
  Math \textbf{4} (1960), 187--209.

\bibitem[Bil61]{Bil61}
\bysame, \emph{Hausdorff dimension in probability theory ii}, Illinois J. Math
  \textbf{5} (1961), 291--298.

\bibitem[BMP92]{BMP92}
G.~Brown, G.~Michon, and J.~Peyri\`ere, \emph{On the multifractal analysis of
  measures}, J. Stat. Phys. \textbf{66} (1992), 775--790.

\bibitem[BS10]{BuSe10}
Z.~Buczolich and S.~Seuret, \emph{Typical measures on $[0,1]^d$ satisfy a
  multifractal formalism}, Nonlinearity \textbf{23} (2010), 2905--2918.

\bibitem[CN10]{CN10}
M.~Clausel and S.~Nicolay, \emph{{Some prevalent results about strongly
  monoH\"older functions}}, Nonlinearity \textbf{23} (2010), 2101--2116.

\bibitem[Cut95]{Cut95}
C.D. Cutler, \emph{{Strong and weak duality principles for fractal dimension in
  Euclidean space}}, Math. Proc. Camb. Phil. Soc. \textbf{118} (1995),
  393--410.

\bibitem[DS97]{DS97}
G.~David and S.~Semmes, \emph{{Fractured fractals and broken dreams.
  Self-similar geometry through metric and measure}}, vol.~7, Oxford University
  Press, 1997.

\bibitem[Fal85]{Fal85}
K.~Falconer, \emph{The geometry of fractal sets}, vol.~85, Cambridge University
  Press, 1985.

\bibitem[Fal97]{Fal97}
\bysame, \emph{{Techniques in Fractal Geometry}}, Wiley, 1997.

\bibitem[FH]{FR13}
J.~Fraser and J.T. Hyde, \emph{{The Hausdorff dimension of graphs of prevalent
  continuous functions}}, Real Analysis Exchange \textbf{(to appear)}.

\bibitem[FJ06]{FJ08}
A.~Fraysse and S.~Jaffard, \emph{How smooth is almost every function in a
  {S}obolev space?}, Rev. Mat. Iberoamericana \textbf{22} (2006), 663--683.

\bibitem[Gen98]{Gen97}
J.~Genyuk, \emph{A typical measure typically has no local dimension}, Real
  Anal. Exchange \textbf{23} (1997/98), 525--537.

\bibitem[Haa92]{Ha92}
H.~Haase, \emph{A survey on the dimension of measures}, Topology, Measures, and
  Fractals, 1992, pp.~66--75.

\bibitem[JP95]{JP95}
H.~Joyce and D.~Preiss, \emph{On the existence of subsets of finite positive
  packing measure}, Mathematika \textbf{42} (1995), 15--24.

\bibitem[Mat95]{Mat95}
P.~Mattila, \emph{{Geometry of Sets and Measures in Euclidean Spaces}},
  vol.~44, Cambridge University Press, 1995.

\bibitem[MP10]{MP10}
P.~M\"orters and Y.~Peres, \emph{Brownian motion}, Cambridge University Press,
  2010.

\bibitem[MWW02]{MWW02}
J.-H. Ma, Z.-Y. Wen, and J.~Wu, \emph{{Besicovitch subsets of self-similar
  sets}}, Annales de l'institut Fourier \textbf{52} (2002), 1061--1074.

\bibitem[Ols95]{Ol95}
L.~Olsen, \emph{A multifractal formalism}, Adv. Math. \textbf{116} (1995),
  9--195.

\bibitem[Ols05]{Ol05}
\bysame, \emph{{Typical $L^q$-dimensions of measures}}, Monatsh. Math.
  \textbf{146} (2005), 143--157.

\bibitem[Ols10a]{Ol10b}
\bysame, \emph{Fractal and multifractal dimensions of prevalent measures},
  Indiana Univ. Math. J. \textbf{59} (2010), 661--690.

\bibitem[Ols10b]{Ol10}
\bysame, \emph{{Prevalent $L^q$-dimensions of measures}}, Math. Proc. Camb.
  Phil. Soc. \textbf{149} (2010), 553--571.

\bibitem[OW03]{OW03}
L.~Olsen and S.~Winter, \emph{Normal and non-normal points of self-similar sets
  and divergence points of self-similar measures}, J. London Math. Soc.
  \textbf{67} (2003), 103--122.

\bibitem[OY05]{OY05}
W.~Ott and J.~Yorke, \emph{Prevalence}, Bull. Amer. Math. Soc. \textbf{42}
  (2005), 263--290.

\bibitem[Par67]{Par67}
K.R. Parthasarathy, \emph{Probability measures on metric spaces}, Probability
  and Mathematical Statistics, Academic Press, 1967.

\bibitem[Tri82]{Tri82}
C.~Tricot, \emph{Two definitions of fractional dimension}, Math. Proc. Camb.
  Phil. Soc. \textbf{91} (1982), 57--74.

\end{thebibliography}

\providecommand{\bysame}{\leavevmode\hbox to3em{\hrulefill}\thinspace}
\providecommand{\MR}{\relax\ifhmode\unskip\space\fi MR }
\providecommand{\MRhref}[2]{%
  \href{http://www.ams.org/mathscinet-getitem?mr=#1}{#2}
}
\providecommand{\href}[2]{#2}

\end{document}